\documentclass[12pt, british]{article}
\usepackage{color}
\usepackage{babel}
\usepackage{amsmath,amsthm}
\usepackage{amsfonts}
\usepackage{amssymb}
\usepackage{graphicx}
\usepackage{pdfpages}
\usepackage[utf8]{inputenc}
\usepackage{multicol}
\usepackage[all]{xy}
\usepackage{enumerate}
\usepackage{wrapfig}
\usepackage{geometry}
\usepackage{array}
\newtheorem{theorem}{Theorem}[section]

\newtheorem{corollary}[theorem]{Corollary}

\newtheorem{lemma}[theorem]{Lemma}
\newtheorem{proposition}[theorem]{Proposition}

\theoremstyle{definition}

\newtheorem{definition}[theorem]{Definition}

\newtheorem{remark}[theorem]{Remark}

\newcommand{\z}{\mathbb{Z}}
\newcommand{\n}{\mathbb{N}}

\usepackage[
bookmarks=true,
breaklinks=true,
bookmarksnumbered = true,
colorlinks= true,
urlcolor= green,
anchorcolor = yellow,
citecolor=blue,
]{hyperref}                   

\begin{document}

\title{A Nielsen type periodic number for maps over $B$}
\author{WESLEM LIBERATO SILVA
~\footnote{Departamento de Ciências Exatas, Universidade Estadual de Santa Cruz, Rodovia Jorge Amado, Km 16, Bairro Salobrinho, CEP 45662-900, Ilhéus-BA, Brazil.
e-mail: \texttt{wlsilva@uesc.br}}
\and
RAFAEL MOREIRA DE SOUZA
~\footnote{Universidade Estadual de Mato Grosso do Sul, 
Cidade Universit\'aria de Dourados - Caixa postal 351 - CEP: 79804-970, Dourados-MS, Brazil.
e-mail: \texttt{moreira@uems.br}}
}
%
\maketitle


\begin{abstract}

Let $Y \to E \stackrel{p}{\to} B$ be a fibration and let $f: E \to E$ be a fiber map over $B$. In this work, we study the geometric and algebraic Reidemeister classes of the iterates of $f$ and introduce a  Nielsen-type periodic number over $B$, denoted by $N_B P_n(f)$. When $B$ is a point, then $N_B P_n(f)$ coincides with the classical Nielsen periodic number.

%
%
%
%
%
\end{abstract}

\maketitle

\noindent


\section{Introduction}


Let $f: X \to X$ be a self-map of a topological space and $n$ a positive integer. A point $x \in X$ is called a {\it $n$-periodic point} of $f$ if $x \in Fix(f^{n}),$ where $Fix(f^{n}) = \{x \in X| f^{n}(x) = x \}$. The study of periodic points plays an important role in dynamics, as it primarily focuses on the behavior of the orbits of a map $f,$ that is, sets of the form: $\{f^{n}(x)| 1 \leq n < \infty \}.$ A natural generalization involves studying periodic points for fiber-preserving maps motivated by the study of extensions of dynamical systems.
A dynamical system $(M, D)$ is called an {\it extension} of a base dynamical system $(B, d)$ if there exists a continuous surjective map $p: E \to B,$ called projection map, such that $p \circ D = d \circ p.$ The system $(B, d)$ is called a factor of $(E, D).$ For more details on this definition, see \cite{K-S-T-14}. We denote
$$P_{n}(f) = \{x \in X |f^{n}(x) = x \,\, \textrm{but} \,\, f^{k}(x) \neq x \,\, \textrm{for any} \,\,  k < n \}. $$

In \cite{H-P-Y} was defined a Nielsen type periodic number, denoted by $NP_n(f),$ for a map $f: X \to X$ on a compact, connected and ANR space. This number is a homotopic invariant and satisfies $NP_n(f) \leq min\{\#P_n(g)|  g \sim f \}.$ 

Let $Y \to E \stackrel{p}{\to} B$ be a fibration and $f: E \to E$ be a fiber map over $B,$ that is, $p \circ f = p.$ All spaces here are assumed to be compact, without boundary, path-connected manifolds. Under these hypothesis, a Nielsen theory for $f$ was developed in \cite{G-K-09} using bordism techniques. In this paper, we used this theory to introduce a Nielsen type periodic number for a map $f$ over $B$, denoted by $N_BP_n(f)$. This number is a homotopy invariant over $B$ and satisfies $$N_B P_n(f) \leq  min\{\# \pi_0(P_n(g))|  g \sim_B f \},$$ where the symbol $\sim_B$ means a homotopy over $B.$  

By the definition of the Reidemeister classes over $B$, it is possible to see that $N_BP_n(f)$ coincides with $NP_n(f)$ when $B$ is a point. Even in the case of trivial fibrations, computing $N_BP_n(f)$ is not easy. Motivated by the classical case, we define when $f$ is a $n$-toral map over $B$ and present a formula for $N_BP_n(f),$ as stated in Theorem \ref{th-nbpn}.

Informations about $N_BP_n(f)$ implies directly in the computation of the minimal periods of maps over $B.$ In fact, we consider $Per(f) = \{n \in \mathbb{N}| P_n(f) \neq \emptyset \}$ that is not an invariant by homotopy over $B$ and $H_{B}Per(f) = \displaystyle\bigcap_{g \simeq_{B} f} Per(g)$ is invariant over $B.$  Thus, if $N_BP_n(f) \neq 0$, then $n$ belongs to $H_{B}Per(f).$ The set $H_{B}Per(f)$ has been explored by many authors in the case where $B$ is point, see, for example, \cite{A-B-L-S-S-95}, \cite{J-L-98} and \cite{K-K-Z}. We intent to compute $H_{B}Per(f)$ in future work using $N_BP_n(f)$ and other techniques, as in general $N_BP_n(f)$ is difficult to compute.

Besides the introduction, this paper is organized into four sections. In Section 2, we study the Nielsen and Reidemeister classes for a map $f$ over $B.$ In Section 3, we study the algebraic and geometric periodic classes over $B$ and define the depth and the length of a periodic class. In Section 4, we define the Nielsen type periodic number $N_BP_n(f).$ In Section 5, we focus on computing $N_BP_n(f)$ for some fibrations.

\section{Nielsen and Reidemeister classes over $B$}

Let $Y \to E \stackrel{p}{\to} B$ be a fibration and $f: E \to E$ a fiber map over $B,$ where all spaces are compact, without boundary, path-connected manifolds. Motivated by \cite{G-K-09}, in this section, we study the Nielsen and Reidemeister classes of $f$ over $B.$ There is a subtle difference between the definition given in \cite{G-K-09} and ours: in \cite{G-K-09}, its is assumed that $f$ has a fixed point, whereas we do not make this assumption.

\begin{definition}[Algebraic Reidemeister classes over $B$] \label{def-rb}

Let $x_0 $ be a point in $E$ and $\omega$ a path from $x_0$ to $f(x_0)$ such that $p(\omega(t)) = p(x_0)$ for all $t \in I.$  We consider $Y \subset E$ the fiber over the point $b_0 = p(x_0).$  Two elements $[\theta]$ and $[\theta^{'}]$ in $\pi_1(Y, x_0)$ are Reidemeister related over $B$ if and only if there exists $[c] \in \pi_1(E, x_0)$ such that 
$$[c] \ast_B [\theta] = [\theta^{'}].$$ The action $\ast_B$ is defined as follows. Since $p$ is a fibration then from \cite[I.7.16]{Whi} we can lift the homotopy $H: I \times I \to B$ defined by $H(t,s) = p(c(t))$ to a homotopy $\widetilde{H}: I \times I \to E$ such that
$$\widetilde{H}(t,0) = c(t), \,\,\,\, \widetilde{H}(0,s) = \theta(s) \ast \omega(s),  \,\,\,\,  \widetilde{H}(t,1) = f(c(t)) $$  
for all $t,s \in I.$ We define

\begin{equation} \label{actionastB}
	[c] \ast_B [\theta] = [\widetilde{H}(1,s) \ast \omega^{-1}] \in \pi_1(Y, x_0).
\end{equation}

The action is represented in the picture below, where each edge label represents the image of that edge of $I \times I$ by the homotopy $\widetilde{H}$ in $E.$
\begin{center}
\setlength{\unitlength}{0,8cm}
\begin{picture}(6.7,3.3)
\thicklines
\put(1.5,0.5){\line(1,0){3}}
\put(1.5,2.5){\line(1,0){3}}
\put(1.5,0.5){\line(0,1){2}}
\put(4.5,0.5){\line(0,1){2}}
\put(2.9,0){$c$}
\put(2.5,2.8){$f(c)$}
\put(0.1,1.4){$\theta \ast \omega$}
\put(4.8,1.4){$\theta^{'} \ast \omega$}
\end{picture}
\end{center}

The algebraic Reidemeister classes over $B,$ denoted by $\mathcal{R}_{B}(f; x_0, \omega),$ is the set of classes given by the equivalence relation  $\ast_B.$ 
\end{definition}

Due to the specific form of $H$ and the fact that $f$ is over $B$, we have that $\widetilde{H}(1,s) \subset Y,$ where $Y = p^{-1}(p(x_0)).$ Since $\omega$ is a path in $Y$ then $\widetilde{H}(1,s) \ast \omega^{-1} \in \pi_1(Y, x_0).$ By \cite[I.7.18]{Whi}, the action $[c] \ast_{B} [\theta]$ depends only of the homotopy classes of $c,$  $\theta$ and the path $\omega$.

%
%
%
%
%
%

\begin{remark}
For simplicity, from now on, we will use diagrams like above to represent the lift $\widetilde{H}$ of a homotopy $H,$ where each edge label represents the image of that edge by $\widetilde{H}.$
\end{remark}

It follows from boundary conditions that $\theta^{'} \ast \omega$ is homotopic to $c^{-1} \ast \theta \ast \omega \ast f(c)$ in $E.$ Thus, in the special  case where $B$ is a point, we have $[\theta^{'}] = [c^{-1} \ast \theta \ast (\omega \ast f(c) \ast \omega^{-1})]$ and therefore the action $\ast_{B}$ coincides with that given in \cite{H-P-Y}.

\begin{lemma} \label{bij-x0-x1}
Let $x_0$ and $ x_1$ be points in $E$ connected by a path $u$ in $E$ from $x_0$ to $x_1$. We take  $\omega_0$ and $\omega_1$ paths in the fiber $Y_0 = p^{-1}(p(x_0))$ from $x_0$ to $f(x_0)$ and in the fiber $Y_1 = p^{-1}(p(x_1))$ from $x_1$ to $f(x_1)$, respectively. Then exists a map $\kappa:\mathcal{R}_{B}(f; x_0, \omega_0) \to \mathcal{R}_{B}(f; x_1, \omega_1)$ which is a bijection. Furthermore, $\kappa$ does not depend on the choice of the path $u$.      
\end{lemma}

\begin{proof} 
Given $[\theta] \in \mathcal{R}_{B}(f; x_0, \omega_0),$
from \cite[I.7.16]{Whi} we can lift the homotopy $H_1: I \times I \to B$ defined by $H_1(t,s) = p(u(t))$ to a homotopy $\widetilde{H}_{1}: I \times I \to E$ such that: $\widetilde{H}_{1}(t,0) = u(t), \ \widetilde{H}_{1}(0,s) = \theta(s)\ast \omega_{0}(s), \ \widetilde{H}_{1}(t,1) = f(u(t)) \ and \ \widetilde{H}_{1}(t,1) \in Y_{1},$
for all $t,s \in I,$ as in diagram below.

\begin{center}
\setlength{\unitlength}{0,8cm}
\begin{picture}(6.7,3.3)
\thicklines
\put(1.5,0.5){\line(1,0){3}}
\put(1.5,2.5){\line(1,0){3}}
\put(1.5,0.5){\line(0,1){2}}
\put(4.5,0.5){\line(0,1){2}}
\put(2.9,0){$u$}
\put(2.5,2.8){$f(u)$}
\put(0.1,1.4){$\theta\ast\omega_{0}$}
\put(4.6,1.4){$\widetilde{H}_{1}(1,s)$}
\end{picture}
\end{center}

We define $\kappa([\theta])=[\widetilde{H}_{1}(1,s)\ast\omega_{1}^{-1}].$
Given another class $[\theta^{'}] \in \mathcal{R}_{B}(f; x_0, \omega_0)$ we have $\kappa([\theta^{'}])=[\widetilde{H}_{2}(1,s)\ast\omega_{1}^{-1}]$ for some lift $\widetilde{H}_2.$ If $[\theta] = [\theta^{'}],$ then there exists $[c] \in \pi_1(E, x_0)$ and a lift $\widetilde{H}: I \times I \to E$ for the homotopy $H: I \times I \to B$ defined by $H(t,s) = p(c(t))$ such that the image of the boundary of $I \times I$ is represented as in the diagram below, by Definition \eqref{def-rb}. 
\begin{center}
\setlength{\unitlength}{0,8cm}
\begin{picture}(6.7,3.3)
\thicklines
\put(1.5,0.5){\line(1,0){3}}
\put(1.5,2.5){\line(1,0){3}}
\put(1.5,0.5){\line(0,1){2}}
\put(4.5,0.5){\line(0,1){2}}
\put(2.9,0){$c$}
\put(2.5,2.8){$f(c)$}
\put(0.1,1.4){$\theta\ast\omega_{0}$}
\put(4.6,1.4){$\theta^{'}\ast\omega_{0}$}
\end{picture}
\end{center}

By the following concatenation of the homotopies; $\widetilde{H}_1(1-t,s) \ast \widetilde{H}(t,s) \ast \widetilde{H}_2(t,s)$ we have that $[\widetilde{H}_1(1,s)] = [\widetilde{H}_2(1,s)],$ as in the diagram below.
\begin{center}
\setlength{\unitlength}{0,8cm}
\begin{picture}(13,3.3)
\thicklines
\put(2,0.5){\line(1,0){9}}
\put(2,2.5){\line(1,0){9}}
\put(2,0.5){\line(0,1){2}}
\put(5,0.5){\line(0,1){2}}
\put(8,0.5){\line(0,1){2}}
\put(11,0.5){\line(0,1){2}}
\put(3.1,0){$u^{-1}$}
\put(2.7,2.8){$f(u^{-1})$}
\put(6.4,0){$c$}
\put(6.2,2.8){$f(c)$}
\put(9.4,0){$u$}
\put(9.2,2.8){$f(u)$}
\put(0.1,1.4){$\widetilde{H}_{1}(1,s)$}
\put(5.1,1.4){$\theta\ast\omega_{0}$}
\put(8.1,1.4){$\theta^{'}\ast\omega_{0}$}
\put(11.1,1.4){$\widetilde{H}_{2}(1,s)$}
\end{picture}
\end{center}

Therefore $\kappa([\theta]) = [\widetilde{H}_{1}(1,s)\ast\omega_{1}^{-1}]  = [\widetilde{H}_{2}(1,s)\ast\omega_{1}^{-1}] = \kappa([\theta^{'}])$ and thus $\kappa$ is well defined.
To demonstrate that $\kappa$ is a bijection we define $\widetilde{\kappa}:\mathcal{R}_{B}(f; x_1, \omega_1) \to \mathcal{R}_{B}(f; x_0, \omega_0)$ by the same way we defined $\kappa$. Given $[\theta]\in\pi_1(Y_0,x_0)$ we have:
\begin{center}
\setlength{\unitlength}{0,8cm}
\begin{picture}(9.5,3.3)
\thicklines
\put(1.5,0.5){\line(1,0){6}}
\put(1.5,2.5){\line(1,0){6}}
\put(1.5,0.5){\line(0,1){2}}
\put(4.5,0.5){\line(0,1){2}}
\put(7.5,0.5){\line(0,1){2}}
\put(2.8,0){$u$}
\put(2.4,2.8){$f(u)$}
\put(5.7,0){$u^{-1}$}
\put(5.4,2.8){$f(u^{-1})$}
\put(0.1,1.4){$\theta\ast\omega_{0}$}
\put(4.6,1.4){$\widetilde{H}_{1}(1,s)$}
\put(7.6,1.4){$\widetilde{H}_{2}(1,s)$}
\end{picture}
\end{center}
and $[\theta]=[\widetilde{H}_{2}(1,s)\ast\omega_{0}^{-1} ]$, then $\widetilde{\kappa}(\kappa([\theta]))=[\theta]$.
Finally, let $v$ another path in $E$ from $x_0$ to $x_1$, then $\kappa_{v}:\mathcal{R}_{B}(f; x_0, \omega_0) \to \mathcal{R}_{B}(f; x_1, \omega_1)$ is defined in a similar way to $\kappa_{u}$. Given $[\theta]\in\pi_1(Y_0,x_0)$ we have:
\begin{center}
\setlength{\unitlength}{0,8cm}
\begin{picture}(10,3.3)
\thicklines
\put(2,0.5){\line(1,0){6}}
\put(2,2.5){\line(1,0){6}}
\put(2,0.5){\line(0,1){2}}
\put(5,0.5){\line(0,1){2}}
\put(8,0.5){\line(0,1){2}}
\put(3.1,0){$u^{-1}$}
\put(2.7,2.8){$f(u^{-1})$}
\put(6.4,0){$v$}
\put(6.1,2.8){$f(v)$}
\put(0.2,1.4){$\widetilde{H}_{1}(1,s)$}
\put(5.1,1.4){$\theta\ast\omega_{0}$}
\put(8.1,1.4){$\widetilde{H}_{2}(1,s)$}
\end{picture}
\end{center}
and $\kappa_{u}([\theta])=[\widetilde{H}_{1}(1,s)\ast\omega_{1}^{-1}]=[\widetilde{H}_{2}(1,s)\ast\omega_{1}^{-1}]=\kappa_{v}([\theta])$.
\end{proof}

\begin{definition}
The cardinality of $\mathcal{R}_{B}(f; x_0, \omega)$ is the Reidemeister number and it will be denote by $R_B(f).$
\end{definition}

\begin{corollary}	
The Reidemeister number depends only of the homotopy classes of $f$ over $B.$
\end{corollary}


We denote by $E_B(f)$ the following space;
$$E_B(f) = \{(x, \alpha) \in E \times E^{I}| p(\alpha(t)) = p(x), \alpha(0) = x , \alpha(1) = f(x)\},$$
where $E^{I}$ is the space of all continuous paths on $E$ with the compact-open topology. The map $p_1: E_B(f) \to E$ defined by $p_1(x, \alpha) = x$ is a fibration. We denote $p_r = p \circ p_1,$ which is also a fibration.

\begin{theorem} \label{th-rb}
Let $f: E \to E$ be a map over $B.$ Let $x_0 \in E$ and $\omega$ a path from $x_0$ to $f(x_0)$ such that $p(\omega(t)) = p(x_0)$ for all $t \in I.$ There is a canonical bijection 
$$\Gamma: \mathcal{R}_{B}(f; x_0, \omega) \to \pi_0(E_B(f)).$$
\end{theorem}

\begin{proof}
Let $[\theta] \in \pi_1(Y, x_0)$ be a class where $Y = p^{-1}(p(x_0)).$ The pair $(x_0, \theta \ast \omega) \in E_B(f).$ Define 
$$\Gamma([\theta]) = \mathcal{C}_{\theta},$$
where $\mathcal{C_{\theta}}$ denotes the path component of $(x_0,\theta \ast \omega ).$  If $[\theta^{'}] = [c] \ast_B [\theta]$ then the homotopy $\widetilde{H}$ in \eqref{actionastB} determines a path $(c(t), \widetilde{H}(t, -))$ joining $(x_0, \theta \ast \omega)$ to $(x_0, \theta^{'} \ast \omega)$ in $E_B(f),$ and therefore $\Gamma([\theta]) = \Gamma([\theta^{'}]).$ Is not difficult to see that two paths in $E_B(f)$ which starts and end in the fiber $p_1^{-1}(x_0) = (x_0, \Omega(Y, x_0))$ can be joining by some lift $\widetilde{H}$ as in Definition \eqref{def-rb}. Thus $\Gamma$ is injective.  

The map $\Gamma$ is also onto. In fact, given a point $(x, \alpha_1) \in E_B(f),$  we take a path $r$ in $E$ from $x$ to $x_0$ and lift it to a path in $E_B(f)$ which joins $(x, \alpha_1)$ to some point in $(x_0, \alpha) \in  p_1^{-1}(x_0).$ Therefore $\Gamma([\alpha \ast \omega^{-1}]) = \mathcal{C_{\alpha}} = \mathcal{C}_{\alpha_1}.$
\end{proof}

The set of path components of $E_B(f)$ is called by geometric Reidemeister set over $B$ and its cardinality by geometric Reidemeister number over $B.$

\begin{definition}[Nielsen classes over $B$] \label{nielsen-b}
Let $f: E \to E$ be a map over $B.$ Two points $x, y \in Fix(f)$ are called Nielsen equivalent over $B$ if there exist a path $\lambda: I \to E$ with $\lambda(0) = x$ and $\lambda(1) = y$ and a homotopy $H: I \times I \to E$ such that $H(t,0) = \lambda(t),$ $H(t,1) = f\circ \lambda(t),$ $H(0,s) = x,$ $H(1,s) = y$ and for each $t$ the image $H(\{t\} \times I)$ lies in the fiber $p^{-1}(p(\lambda(t))).$   
\end{definition}

Using the equivalence relation above, we can split $Fix(f)$ into disjoint Nielsen classes over $B.$ Each Nielsen class is associated with an index over $B$, as in \cite[Section 5]{G-K-09}. The Nielsen number over $B$ is the number of Nielsen classes with nonzero index. We denote by $\mathcal{N}_B(f)$ the set of Nielsen classes over $B.$

\begin{proposition} \label{prop-fix}

Two fixed points $x_1$ and $x_2$ of $f: E \to E$ are Nielsen equivalent over $B$, if and only if, the map
$$g_B : Fix(f) \to E_B(f)$$
defined by $g_B(x) = (x, \textrm{constant path at } \, f(x)=x)$ takes them into the same path-component $A$ of $E_B(f).$ Therefore the Nielsen classes of $f$ over $B$ are just those inverse images $g_B^{-1}(A),$ $A \in \pi_0(E_B(f))$ which are nonempty.
\end{proposition}

\begin{proof}

Follows from \cite[Proposition 4.1]{G-K-09}.
\end{proof}

\begin{proposition} \label{prop-phi}
There is a map $\phi: \mathcal{N}_{B}(f) \to \mathcal{R}_{B}(f; x_0, \omega)$ which is an injection.   
\end{proposition}
\begin{proof}
Given $x$ a fixed point of $f$ we define $\phi([x]) = \Gamma^{-1}(C_{(x, c_x)}),$ where $c_x$ denotes the constant path at $x = f(x)$ and $\Gamma$ is defined in Theorem \ref{th-rb}. The Definition \eqref{nielsen-b} implies that two points $x, y \in Fix(f)$ belong to the same fixed point class over $B$ if and only if $(x, c_x)$ and $(y, c_y)$ belong to the same path-component of $E_B(f).$ Therefore $\phi$ is well defined and injective.
\end{proof}

\section{Geometric and algebraic periodic point classes over $B$}


Given $f: E \to E$ a map over $B$ and $n$ a positive integer, in this section we will study the Nielsen and Reidemeister classes of $f^{n}.$ 

Let $\mathcal{F}_n  \in \mathcal{N}_B(f^{n})$ be a fixed point class and $x \in \mathcal{F}_n.$  
Let $f(\mathcal{F}_n)$ be the unique class defined by $f(x).$ It is not difficult to see that $f(\mathcal{F}_n)$ is independent of the choice of $x.$
Thus, $f$ defines a function from $\mathcal{N}_B(f^{n})$ to $\mathcal{N}_B(f^{n})$ which, by abuse of notation, we also denote by $f.$  
Note that $f^{n}: \mathcal{N}_B(f^{n}) \to \mathcal{N}_B(f^{n})$ is the identity function.

\begin{definition}
	
The $f$-orbit of the periodic class $\mathcal{F}_n$ is the set;
$$\langle  \mathcal{F}_n \rangle  = \{\mathcal{F}_n, f(\mathcal{F}_n), \cdots, f^{l(\langle \mathcal{F}_n \rangle)-1}(\mathcal{F}_n) \},$$
where $l(\langle \mathcal{F}_n \rangle)$ is called the {\it length} of the orbit $\langle  \mathcal{F}_n \rangle ,$ it is the smallest positive integer $r$ such that $f^{r}(\mathcal{F}_n) = \mathcal{F}_n.$ 	
\end{definition}

By the division algorithm we can check that  $l(\langle \mathcal{F}_n \rangle)$ divides $n.$ 

Let $k$ be a positive divisor of $n,$ $\mathcal{F}_k  \in \mathcal{N}_B(f^{k})$ and $x \in \mathcal{F}_k.$ Define $\iota(\mathcal{F}_k)$ as the unique class of $\mathcal{N}_B(f^{n})$ determined by $x.$ This correspondence defines a function $$\iota:\mathcal{N}_B(f^{k}) \to \mathcal{N}_B(f^{n}).$$  

As $\mathcal{F}_k \subset \iota(\mathcal{F}_k)$, if $ \mathcal{F}_n = \iota(\mathcal{F}_k)$ we say that $ \mathcal{F}_k$ is contained in $ \mathcal{F}_n$ and write $\mathcal{F}_k \subset \mathcal{F}_n.$ As in the classical case, in general de map $\iota$ is not injective.

\begin{definition}
	
Let $\mathcal{F}_n$ be a fixed point class of $f^{n}.$ The (geometric) {\it depth}, denoted by $d(\mathcal{F}_n),$ is the smallest integer $k$ for which there exists a periodic point class $\mathcal{F}_k  \in \mathcal{N}_B(f^{k})$ such that $\mathcal{F}_k \subset \mathcal{F}_n.$ 
If the depth of $\mathcal{F}_n$ is $n$ the class  $\mathcal{F}_n$ is said to be {\it irreducible}.  	
\end{definition}

If $k = d(\mathcal{F}_n)$ then any element of the orbit $\langle  \mathcal{F}_n \rangle$ has the same depth $k.$ Thus, depth is a property of orbits. In fact, if $k = d(\mathcal{F}_n)$, then there exists a periodic point class $\mathcal{F}_k \in \mathcal{N}_B(f^{k})$ such that $\mathcal{F}_k \subset \mathcal{F}_n$. So, $f^{i}(\mathcal{F}_k) \subset f^{i}(\mathcal{F}_n)$ and $f^{i}(\mathcal{F}_k) \in \mathcal{N}_B(f^{k})$ because $f^{k}(f^{i}(x))=f^{i}(f^{k}(x))=f^{i}(x)$, for $f^{i}(\mathcal{F}_n)\in\langle  \mathcal{F}_n \rangle$ and $x\in\mathcal{F}_k$ . Thus, $k\geq d(f^{i}(\mathcal{F}_n))$ for $f^{i}(\mathcal{F}_n)\in\langle  \mathcal{F}_n \rangle$. But, if $k>l=d(f^{j}(\mathcal{F}_n))$, then there exists a periodic point class $\mathcal{F}_l \in \mathcal{N}_B(f^{l})$ such that $\mathcal{F}_l \subset f^{j}(\mathcal{F}_n)$ and $\mathcal{N}_B(f^{l})\ni f^{n-j}(\mathcal{F}_l) \subset f^{n}(\mathcal{F}_n)=\mathcal{F}_n$. It is a absurd because $l<k = d(\mathcal{F}_n)$. Therefore, $k=d(f^{i}(\mathcal{F}_n))$, for $f^{i}(\mathcal{F}_n)\in\langle  \mathcal{F}_n \rangle$.

\begin{proposition}

Let $\mathcal{F}_n  \in \mathcal{N}_B(f^{n})$ be a periodic point class with orbit length $l$ and depth $d.$ Then, $l$ divides $d$ and $\mathcal{F}_n $ contains at least $d/l$ points. In particular,	an irreducible periodic point class $\mathcal{F}_n $ of length $l$ has at least $n/l$ periodic points, where each periodic point belongs to different path components.

\end{proposition}
\begin{proof}
Let $\mathcal{F}_d \subset \mathcal{F}_n$ be irreducible. We have that 
$f^{d}: \mathcal{N}_B(f^{d}) \to \mathcal{N}_B(f^{d})$ is the identity, $f^{d}: \mathcal{N}_B(f^{n}) \to \mathcal{N}_B(f^{n})$ is well defined, 
$f^{d}(\mathcal{F}_n) = \mathcal{F}_n$ and so $l | d$ by the division algorithm. Further, for any $x \in \mathcal{F}_d,$ the set $ S = \{x, f(x), \cdots, f^{d-1}(x) \}$ is such that $x \in \mathcal{F}_d$, $f^{i}(x) \in f^{i}(\mathcal{F}_d)$ for $1\leq i\leq d-1$ and it has distinct elements by the minimality of $d.$ Let $d = kl$ and consider the set:
$$ \overline{S} = \{x, f^{l}(x), f^{2l}(x), \cdots, f^{(k-1)l}(x) \}.$$
Clearly $\overline{S} \subset S$ and contains $d/l$ distinct elements. Since $f^{l}(\mathcal{F}_n) \subset \mathcal{F}_n$ then $\overline{S} \subset \mathcal{F}_n.$ By definition of length and from Proposition \ref{prop-fix} each point of $S$ belongs to different path component of $Fix(f^{d}).$  
\end{proof}

Given $f: E \to E$ a map over $B$, our main purpose is to present the definition of the Nielsen periodic number of $f$ over $B.$ For this purpose, we will  study the periodic points of $f$ from an algebraic point of view.
For $n \geq 1,$ let $n(\omega): I \to E$ be the path from $x_0$ to $f^{n}(x_0)$ given by:
$$n(\omega) = \omega \ast f(\omega) \ast \cdots \ast f^{n-1}(\omega).$$ 
Since $\omega \subset Y = p^{-1}(x_0)$ and $f$ is a map over $B$ then also $n(\omega) \subset Y.$ 
Define the homomorphism $f^{n \omega}: \pi_1(Y, x_0) \to \pi_1(Y, x_0)$ by:
$$f^{n \omega}(\alpha) = n(\omega) \ast f^{n}(\alpha) \ast n(\omega)^{-1}.$$
We will denote $f^{1 \omega}$ by $f^{\omega}.$ By abuse of notation we will denote by same symbol, $f^{n},$ the map and its induced homomorphism in the fundamental group.

\begin{lemma} \label{lemma-fw}
For all $n, m, r \geq 1,$ we have:

\begin{enumerate}[(i)]
	
\item $(f^{\omega})^{n} = f^{n \omega},$
	
\item $(n+1)(\omega) = n(\omega) \ast f^{n}(\omega)  = \omega \ast f(n(\omega)),$
	
\item $f^{m \omega} \circ f^{r \omega} = f^{(m+r)\omega}.$
	
\end{enumerate} 

\end{lemma}
\begin{proof}
The proof follows by a simple induction.
\end{proof}

By Definition \eqref{def-rb} we obtain the set of Reidemeister classes of $f^{n}$ denoted by $\mathcal{R}_{B}(f^{n}; x_0, n(\omega)).$ An element $[\theta]_{n} \in \mathcal{R}_{B}(f^{n}; x_0, n(\omega))$ is called of an (algebraic) periodic point class of $f^{n}$ over $B.$

\begin{proposition} \label{prop-fw}

There is a function $$[f^{\omega}]: \mathcal{R}_{B}(f^{n}; x_0, n(\omega)) \to \mathcal{R}_{B}(f^{n}; x_0, n(\omega))$$
such that
\begin{enumerate}[(i)]
	
\item $[f^{\omega}](\theta) = [f^{\omega}(\theta)],$ 
	
\item $[f^{\omega}] \circ j = j \circ f^{\omega},$ where $j: \pi_1(Y, x_0) \to \mathcal{R}_{B}(f^{n}; x_0, n(\omega))$ is the natural projection,

\item  $[f^{\omega}]^{n}$ is the identity.
\end{enumerate}
\end{proposition}
\begin{proof}
Given $\theta \in \pi_1(Y, x_0)$ from \cite[I.7.16]{Whi} we can lift the homotopy $H: I \times I \to B$ defined by $H(t,s) = p(\omega^{-1}(t))$ to a homotopy $\widetilde{H}: I \times I \to E$ such that
$$\widetilde{H}(t,0) = \omega^{-1}(t), \,\,\,\, \widetilde{H}(0,s) = f(\theta(s) \ast n(\omega(s))),  \,\,\,\,  \widetilde{H}(t,1) = f^{n}(\omega^{-1}(t)) $$  
for all $t,s \in I.$ We define
\begin{equation} \label{def-fw}
[f^{\omega}](\theta) = [\widetilde{H}(1,s) \ast n(\omega)^{-1}] \in \pi_1(Y, x_0).
\end{equation}

Given another element $\theta^{'} \in \pi_1(Y, x_0) $ we have $[f^{\omega}](\theta^{'}) = [\widetilde{H}_1(1,s) \ast n(\omega)^{-1}]$  for some homotopy $\widetilde{H}_1$ satisfying the conditions of Definition \eqref{def-fw}. 

We suppose that $[\theta] = [\theta^{'}]$ in $\mathcal{R}_{B}(f^{n}; x_0, n(\omega)).$ By Equation \eqref{actionastB}  there are $[c] \in \pi_1(E, x_0)$ and a homotopy  $\widetilde{H}_2: I \times I \to E$ such that $\widetilde{H}_2(t,0) = c(t), \,\,\,\, \widetilde{H}_2(0,s) = \theta(s) \ast n(\omega(s)),  \,\,\,\,  \widetilde{H}_2(t,1) = f^{n}(c(t))$ and $\widetilde{H}_2(1,s) = \theta^{'}(s) \ast n(\omega)(s).$ We consider the homotopy $\widetilde{H}_3(t,s) = f \circ \widetilde{H}_2(t,s).$ 
The homotopy $\widetilde{H}_4(t,s)$ defined by the concatenation of homotopies;
$$ 
\widetilde{H}_4(t,s) = \widetilde{H}(1-t,s) \ast \widetilde{H}_3(t,s) \ast \widetilde{H}_1(t,s)
$$ 
satisfies the following conditions; $\widetilde{H}_4(t,0) = f^{\omega}(c(t)),$ 
$\widetilde{H}_4(t,1) = f^{n}(f^{\omega}(c(t))),$ $\widetilde{H}_4(0,s) =  \widetilde{H}(1,s)$ and $\widetilde{H}_4(1,s) =  \widetilde{H}_1(1,s).$ Thus, $[f^{\omega}](\theta) $ and $ [f^{\omega}](\theta^{'})$ belong to the same path component of $E_B(f^{n}).$ From Theorem \ref{th-rb} we have $[f^{\omega}](\theta) = [f^{\omega}](\theta^{'})$ in $\mathcal{R}_{B}(f^{n}; x_0, n(\omega)),$ and therefore $[f^{\omega}]$ is well defined.
\begin{center}
	\setlength{\unitlength}{0,8cm}
	\begin{picture}(13,3.3)
		\thicklines
		\put(2,0.5){\line(1,0){9}}
		\put(2,2.5){\line(1,0){9}}
		\put(2,0.5){\line(0,1){2}}
		\put(5,0.5){\line(0,1){2}}
		\put(8,0.5){\line(0,1){2}}
		\put(11,0.5){\line(0,1){2}}
		\put(3.1,0){$w$}
		\put(2.7,2.8){$f^{n}(w)$}
		\put(6.4,0){$f(c)$}
		\put(5.8,2.8){$f^{n}(f(c))$}
		\put(9.4,0){$w^{-1}$}
		\put(8.8,2.8){$f^{n}(w^{-1})$}
		\put(0.1,1.4){$\widetilde{H}(1,s)$}
		\put(5.1,1.4){$f(\theta\ast n(\omega))$}
		\put(8.1,1.4){$f(\theta^{'}\ast n(\omega))$}
		\put(11.1,1.4){$\widetilde{H}_{1}(1,s)$}
	\end{picture}
\end{center}	

Item $(i).$ Note that the image of $\widetilde{H}$ is contained in $Y$ because $p \circ \widetilde{H} = H = p \circ \omega^{-1} = p(x_0)$ since  $\omega$ is a path in $Y.$ From the boundary conditions we have the following, in $\pi_1(Y, x_0);$
$$
\begin{array}{lll}
[f^{\omega}](\theta) & = & [\omega \ast f(\theta) \ast f(n(\omega)) \ast f^{n}(\omega^{-1}) \ast n(\omega)^{-1}] \\
& = & [(\omega \ast f(\theta) \ast \omega^{-1}) \ast (\omega \ast f(n(\omega)))  \ast (n+1)(\omega)^{-1}]  \\ 
& = & [(\omega \ast f(\theta) \ast \omega^{-1}) \ast (n+1)(\omega)  \ast (n+1)(\omega)^{-1}]  \\
& = & [\omega \ast f(\theta) \ast \omega^{-1}]  \\
& = & [f^{\omega}(\theta)].  \\
\end{array}
$$ 
From this, we obtain $[f^{\omega}]^{n}(\theta)= [f^{n\omega}(\theta)].$ The item $(ii)$ follows from definition and item $(i).$ 

Item $(iii).$ Using the Definition \eqref{def-rb} with $f^{n}$ and $n(\omega)$ we have $[\theta] \ast_{B} [\theta] = [f^{n\omega}(\theta)].$ Therefore, in $\mathcal{R}_{B}(f^{n}; x_0, n(\omega))$ we have $[f^{n\omega}(\theta)] = [\theta],$ that is, $[f^{\omega}]^{n}$ is the identity function.
\end{proof}

\begin{definition}
	
The (algebraic) orbit of the element $[\theta]_{n}  \in \mathcal{R}_{B}(f^{n}; x_0, n(\omega))$ is the set;
$$\langle  [\theta]_{n} \rangle  = \{[\theta]_{n}, [f^{\omega}]([\theta]_{n}), \cdots, [f^{\omega}]^{l(\langle [\theta]_{n} \rangle)-1}([\theta]_{n}) \},$$
where $l(\langle [\theta]_{n} \rangle),$ called the {\it length} of the orbit $\langle  [\theta]_{n} \rangle ,$ is the smallest positive integer $r$ such that $[f^{\omega}]^{r}([\theta]_{n}) = [\theta]_{n}.$ 
As in the geometric case $l(\langle [\theta]_{n} \rangle)$ divides $n.$  
\end{definition}

\begin{remark} \label{re-orbit}
Note that, by the definition given in Equation \eqref{def-fw}, the length of the orbit of an element $[\theta]_n$ in $\mathcal{R}_{B}(f^{n}; x_0, n(\omega))$ is the same when we consider the orbit of this element $[\theta]_n$ in $\mathcal{R}(f_{|_{Y}}^{n}; x_0, n(\omega)),$ because $[f^{\omega}](\theta) = [f^{\omega}(\theta)]$ in $\pi_1(Y, x_0).$ 
\end{remark}

For $m | n$ we consider the function $\gamma_{m,n}: \pi_1(Y, x_0) \to \pi_1(Y, x_0)$ defined by:
$$ \gamma_{m,n}(\theta) = \theta \ast f^{m \omega}(\theta) \ast f^{2m \omega}(\theta) \ast \cdots \ast f^{(n-m)\omega}(\theta). $$ 
We have:
$$
 \begin{array}{lll}
\gamma_{m,n}(\theta)  & = & \theta \ast m(\omega) \ast f^{m}(\theta) \ast m(\omega)^{-1}  \ast  2m(\omega) \ast f^{2m}(\theta) \ast 2m(\omega)^{-1} \ast \cdots \ast \\ 
 & & (n-m)(\omega) \ast f^{n-m}(\theta) \ast (n-m)(\omega)^{-1}. \\
\end{array}
$$
But,
$$
\begin{array}{lll}
m(\omega)^{-1} \ast 2m(\omega) & = & f^{m-1}(\omega)^{-1} \ast \cdots \ast  f(\omega)^{-1} \ast \omega^{-1} \ast \omega \ast f(\omega) \ast \cdots \ast f^{2m-1}(\omega) \\
& = & f^{m}(\omega) \ast \cdots \ast f^{2m-1}(\omega) \\
& = & f^{m}( \omega \ast \cdots \ast f^{m-1}(\omega) ) \\
& = & f^{m}(m(\omega)). \\	
\end{array}
$$
By induction, we obtain:
$$ im(\omega)^{-1} \ast (i+1)m(\omega) = f^{im}(m(\omega)). $$
From the above relations, we have:
\begin{equation} \label{ex-gamma}
\begin{array}{lll}
\gamma_{m,n}(\theta)  & = & (\theta \ast m(\omega)) \ast f^{m}(\theta \ast m(\omega)) \ast f^{2m}(\theta \ast m(\omega)) \ast \cdots \ast f^{n-m}(\theta \ast m(\omega)) \ast n(\omega)^{-1}. \\
\end{array}
\end{equation}

\begin{lemma}
The function $\gamma_{m,n}$ induces a function 
$$[\gamma_{m,n}]:  \mathcal{R}_{B}(f^{m}; x_0, m(\omega)) \to \mathcal{R}_{B}(f^{n}; x_0, n(\omega)) $$
such that if $k | m | n $ then $[\gamma_{m,n}] \circ [\gamma_{k,m}] = [\gamma_{k,n}].$
\end{lemma}
\begin{proof}
We suppose that $[\theta] = [\theta^{'}]$ in $\mathcal{R}_{B}(f^{m}; x_0, m(\omega)).$ Thus there is a lift $\widetilde{H}: I \times I \to E$ of the homotopy $H(t,s)= p(c(t))$ such that
$$\widetilde{H}(t,0) = c(t), \,\,\,\,  \widetilde{H}(t,1) = f^{m}(c(t)), \,\,\,\, \widetilde{H}(0,s) = \theta(s) \ast m(\omega)(s),  \,\,\,\, \widetilde{H}(1,s) = \theta^{'}(s) \ast m(\omega)(s), $$  
for all $t,s \in I.$
We define the homotopies $H_{k} = f^{km} \circ \widetilde{H},$ for $k = 0, \cdots , \frac{n}{m} - 1.$ From Equation \eqref{ex-gamma} the concatenation $$\overline{H} = H_0 \ast H_1 \ast \cdots \ast H_{\frac{n}{m} - 1}$$ satisfies;  
$\overline{H}(t,0) = c(t), \,\,\,\,  \overline{H}(t,1) = f^{n}(c(t)), \,\,\,\, \overline{H}(0,s) = \gamma_{m,n}(\theta(s)) \ast n(\omega)(s),  \,\,\,\, \overline{H}(1,s) = \gamma_{m,n}(\theta^{'}(s)) \ast n(\omega)(s), $ for all $t,s \in I.$ The homotopy $\overline{H}$ is represented in the picture below, where each edge label represents the image of that edge by $\overline{H}$ in $E.$
\begin{center}
\setlength{\unitlength}{0,8cm}
\begin{picture}(16,3.6)
\thicklines
\put(0,0.5){\line(1,0){16}}
\put(0,2.5){\line(1,0){16}}
\put(0,0.5){\line(0,1){2}}
\put(3.5,0.5){\line(0,1){2}}
\put(8,0.5){\line(0,1){2}}
\put(12,0.5){\line(0,1){2}}
\put(16,0.5){\line(0,1){2}}
\put(0.8,-0.2){$\theta^{'} \ast m(\omega)$}
\put(0.8,2.8){$\theta \ast m(\omega)$}
\put(4.4,-0.2){$f^{m}(\theta^{'} \ast m(\omega))$}
\put(4.4,2.8){$f^{m}(\theta \ast m(\omega))$}
\put(9.7,0){$\cdots$}
\put(9.7,1.5){$\cdots$}
\put(9.7,2.8){$\cdots$}
\put(12.1,-0.2){$f^{n-m}(\theta^{'} \ast m(\omega))$}
\put(12.1,2.8){$f^{n-m}(\theta \ast m(\omega))$}
\put(-0.6,1.4){$c$}
\put(1.9,1.4){$f^{m}(c)$}
\put(6.2,1.4){$f^{2m}(c)$}
\put(12.3,1.4){$f^{n-m}(c)$}
\put(16.3,1.4){$f^{n}(c)$}
\end{picture}
\end{center}

\

Thus $[\gamma_{m,n}(\theta)] = [\gamma_{m,n}(\theta^{'})]$ in $\mathcal{R}_{B}(f^{n}; x_0, n(\omega)),$ and therefore $\gamma_{m,n}$ is well defined. 
\end{proof}

\begin{definition}
	
An element $[\alpha]_{n}  \in \mathcal{R}_{B}(f^{n}; x_0, n(\omega))$ is said to be {\it reducible} to $m$ if there exists $[\beta]_{m}  \in \mathcal{R}_{B}(f^{m}; x_0, m(\omega))$ such that $[\gamma_{m,n}]([\beta]_{m}) =  [\alpha]_{n}.$  In this case, we say that $[\beta]_{m}$ is {\it contained} in $[\alpha]_{n}$ or that $[\alpha]_{n}$ {\it contains} $[\beta]_{m}.$ If $[\alpha]_{n}$ is reducible to $m$ then $m$ divides $n.$ If $[\alpha]_{n}$ is not reducible to $m$ to any $m < n$ then it is said to be {\it irreducible}.    
\end{definition}

\begin{definition}

The (algebraic) {\it depth} of an element $[\alpha]_{n}  \in \mathcal{R}_{B}(f^{n}; x_0, n(\omega))$ is the smallest positive integer $m$ to which $[\alpha]_{n}$ is reducible. We denote the depth of $[\alpha]_{n}$ by $d([\alpha]_{n}).$ Note that $d([\alpha]_{n}) | n$ and for any $n;$ $d([c_{x_0}]_{n}) = 1.$
\end{definition}

The next result shows that depth is a property of orbits.

\begin{lemma} \label{lemma-fwgamma}
If an element $[\alpha]_{n}  \in \mathcal{R}_{B}(f^{n}; x_0, n(\omega))$ is reducible to $m$ then any element in its orbit is also reducible to $m.$
\end{lemma}
\begin{proof}
Using an induction argument we can show that the following diagram;
\begin{equation}
\begin{gathered}
\xymatrix{
\mathcal{R}_{B}(f^{m}; x_0, m(\omega)) \ar[r]^-{[\gamma_{m,n}]} \ar[d]_-{[f^{\omega}]} & \mathcal{R}_{B}(f^{n}; x_0, n(\omega)) \ar[d]^-{[f^{\omega}]} \\
\mathcal{R}_{B}(f^{m}; x_0, m(\omega)) \ar[r]^-{[\gamma_{m,n}]} & \mathcal{R}_{B}(f^{n}; x_0, n(\omega)).
}\end{gathered}
\end{equation}
is commutative. Thus the result follows.
\end{proof}

From above result we can define the {\it depth of the orbit $\langle  [\theta]_{n} \rangle $} by;  $d(\langle  [\theta]_{n} \rangle) = d([\theta]_{n} ).$

\begin{lemma}
The length $l$ of an orbit $\langle  [\theta]_{n} \rangle$ divides its depth $d.$
\end{lemma}
\begin{proof}
We suppose that for $[\alpha]_{n}  \in \mathcal{R}_{B}(f^{n}; x_0, n(\omega))$ there exists $[\beta]_{m}  \in \mathcal{R}_{B}(f^{m}; x_0, m(\omega))$ such that $[\gamma_{m,n}]([\beta]_{m}) =  [\alpha]_{n}.$ By Lemma \eqref{lemma-fwgamma} and the fact that $[f^{\omega}]^{m}$ is the identity on $\mathcal{R}_{B}(f^{m}; x_0, m(\omega))$ we have;
$$[f^{\omega}]^{m}([\alpha]_{n} ) = [f^{\omega}]^{m}([\gamma_{m,n}]([\beta]_{m}) ) = [\gamma_{m,n}]([f^{\omega}]^{m}([\beta]_{m})) = [\gamma_{m,n}]([\beta]_{m}) = [\alpha]_{n} ,$$
and so $l([\alpha]_{n})$ divides $m$ by the division algorithm.
\end{proof}

The next result connects the algebraic e geometric notions of the theory.

\begin{proposition}
For any integers $m$ and $n$ such that $m | n$ the following diagrams commute.
$$
\xymatrix{
\mathcal{N}_{B}(f^{n}) \ar[r]^-{\phi} \ar[d]_-{f} & \mathcal{R}_{B}(f^{n}; x_0, n(\omega)) \ar[d]^-{[f^{\omega}]} \\
\mathcal{N}_{B}(f^{n}) \ar[r]^-{\phi} & \mathcal{R}_{B}(f^{n}; x_0, n(\omega)),
} \hspace{3cm}
\xymatrix{
\mathcal{N}_{B}(f^{m}) \ar[r]^-{\phi} \ar[d]_-{\iota} & \mathcal{R}_{B}(f^{m}; x_0, m(\omega)) \ar[d]^-{[\gamma_{m,n}]} \\
\mathcal{N}_{B}(f^{n}) \ar[r]^-{\phi} & \mathcal{R}_{B}(f^{n}; x_0, n(\omega)).
}
$$
\end{proposition} 
\begin{proof}
Given $\mathcal{F}_n \in \mathcal{N}_{B}(f^{n}) $ and $x \in \mathcal{F}_n$ we can choose a path $c: I \to E$ from $x$ to $x_0.$ 
From \cite[I. 7. 16]{Whi} we can lift the homotopy $H_1: I \times I \to B$ defined by $H_1(t,s) = p(c(t))$ to a homotopy $\widetilde{H}_1: I \times I \to E$ such that
$$\widetilde{H}_1(t,0) = c(t), \,\,\,\, \widetilde{H}_1(0,s) = c_x,  \,\,\,\,  \widetilde{H}_1(t,1) = f^{n}(c(t)) $$  
for all $t,s \in I,$ where $c_x$ denotes the constant path in $x.$ By definition given in the Proposition \eqref{prop-phi} we have
\begin{equation} \label{eq-propmn}
\phi(\mathcal{F}_n) = [\theta] = [\widetilde{H}_1(1,s) \ast n(\omega)^{-1}] \in \mathcal{R}_{B}(f^{n}; x_0, n(\omega)).
\end{equation}
We denote $\widetilde{H}_2 = f \circ \widetilde{H}_1$ and consider the concatenation of homotopies $\widetilde{H}_3 = \widetilde{H}_2 \ast \widetilde{H},$ where $\widetilde{H}$ is given in the proof of Proposition \eqref{prop-fw}. We have; 
$\widetilde{H}_3(t,0) = f(c(t)) \ast \omega^{-1}(t), \,\,\,\, \widetilde{H}_1(0,s) = c_{f(x)},  \,\,\,\,  \widetilde{H}_3(t,1) = f^{n}(f(c(t)) \ast \omega^{-1}(t)) \,\,\, \textrm{and} \,\,\,\,\widetilde{H}_1(0,s) = [f^{\omega}(\theta)], $ like in the diagram below.
\begin{center}
\setlength{\unitlength}{0,8cm}
\begin{picture}(9.5,3.3)
\thicklines
\put(1.5,0.5){\line(1,0){6}}
\put(1.5,2.5){\line(1,0){6}}
\put(1.5,0.5){\line(0,1){2}}
\put(4.5,0.5){\line(0,1){2}}
\put(7.5,0.5){\line(0,1){2}}
\put(2.8,0){$f(c)$}
\put(2.4,2.8){$f^{n}(f(c))$}
\put(5.7,0){$\omega^{-1}$}
\put(5.4,2.8){$f^{n}(\omega^{-1})$}
\put(0.1,1.4){$c_{f(x)}$}
\put(4.6,1.4){$f(\theta \ast n(\omega))$}
\put(7.6,1.4){$f^{\omega}(\theta) \ast n(\omega)$}
\end{picture}
\end{center}

This implies that $(f(x), c_{f(x)})$ and $(x_0, f^{\omega}(\theta) \ast n(\omega) )$ belong to the same path component of $\pi_0(E_B(f^{n})).$  Therefore, $\phi \circ f(\mathcal{F}_n) = [f^{\omega}] \circ \phi(\mathcal{F}_n).$  

For the second diagram we take $\mathcal{F}_m \in \mathcal{N}_{B}(f^{m}), $  $x \in \mathcal{F}_m$ and a path $c: I \to E$ from $x$ to $x_0.$ Let 
$$
\phi(\mathcal{F}_m) = [\theta] = [\widetilde{H}(1,s) \ast m(\omega)^{-1}] \in \mathcal{R}_{B}(f^{m}; x_0, m(\omega)),
$$
as in Equation \eqref{eq-propmn}. For each $j$ we have $f^{jm}(x) = x.$ 
We define the homotopies $H_{k} = f^{km} \circ \widetilde{H},$ for $k = 0, \cdots , \frac{n}{m} - 1.$ From Equation \eqref{ex-gamma} the concatenation $$\overline{H} = H_0 \ast H_1 \ast \cdots \ast H_{\frac{n}{m} - 1}$$ satisfies;  
$\overline{H}(t,0) = c(t), \,\,\,\,  \overline{H}(t,1) = f^{n}(c(t)), \,\,\,\, \overline{H}(0,s) = c_x(s),  \,\,\,\, \overline{H}(1,s) = \gamma_{m,n}(\theta(s)) \ast n(\omega)(s),$ for all $t,s \in I.$ The homotopy $\overline{H}$ is represented in the picture below, where each edge label represents the image of that edge by $\overline{H}$ in $E.$
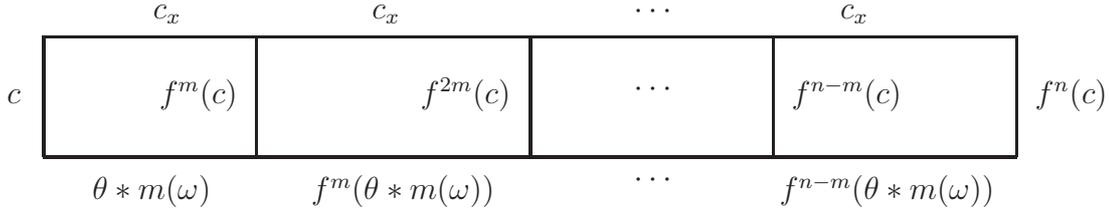
\begin{center}
\setlength{\unitlength}{0,8cm}
\begin{picture}(16,3.6)
\thicklines
\put(0,0.5){\line(1,0){16}}
\put(0,2.5){\line(1,0){16}}
\put(0,0.5){\line(0,1){2}}
\put(3.5,0.5){\line(0,1){2}}
\put(8,0.5){\line(0,1){2}}
\put(12,0.5){\line(0,1){2}}
\put(16,0.5){\line(0,1){2}}
\put(0.8,-0.2){$\theta \ast m(\omega)$}
\put(1.8,2.8){$c_x$}
\put(4.4,-0.2){$f^{m}(\theta \ast m(\omega))$}
\put(5.4,2.8){$c_x$}
\put(9.7,0){$\cdots$}
\put(9.7,1.5){$\cdots$}
\put(9.7,2.8){$\cdots$}
\put(12.1,-0.2){$f^{n-m}(\theta \ast m(\omega))$}
\put(13.1,2.8){$c_x$}
\put(-0.6,1.4){$c$}
\put(1.9,1.4){$f^{m}(c)$}
\put(6.2,1.4){$f^{2m}(c)$}
\put(12.3,1.4){$f^{n-m}(c)$}
\put(16.3,1.4){$f^{n}(c)$}
\end{picture}
\end{center}

\

Therefore, $(x, c_x)$ and $(x_0, \gamma_{m,n}(\theta) \ast n(\omega))$ belong to the same path component of $\pi_0(E_B(f^{n})).$ This implies $\phi \circ \iota(\mathcal{F}_m) = [\gamma_{m,n}] \circ \phi(\mathcal{F}_m).$
\end{proof}

\begin{definition}
The {\it index} of an element $[\alpha]_{n}  \in \mathcal{R}_{B}(f^{n}; x_0, n(\omega))$ is the index of $\mathcal{F}_n$ if 
$\phi(\mathcal{F}_n) = [\alpha]_{n}$ for some $\mathcal{F}_n \in \mathcal{N}_{B}(f^{n}),$ or $0$ otherwise. The elements $[\alpha]_{n}  \in \mathcal{R}_{B}(f^{n}; x_0, n(\omega))$ with non-zero index are said to be 
{\it essential}, otherwise they are called {\it inessential}.
\end{definition}

\section{A Nielsen-type periodic number}

\begin{definition}
Let $O_n(f)$ be the number of irreducible, essential periodic points orbits of $\mathcal{R}_{B}(f^{n}; x_0, n(\omega)).$
The Nielsen type number over $B$ of period $n$ is defined by the formula;
$$N_B P_n(f) = n \times O_n(f).$$
\end{definition}

\begin{proposition}
The following properties of $N_B P_n(f)$ hold true:

\begin{enumerate}[(i)]

\item $N_B P_n(f)$ is a homotopy invariant over $B.$

\item $N_B P_n(f) \leq  MCP_n(f) := min\{\# \pi_0(P_n(g))|  g \sim_B f \}.$

\end{enumerate}

\end{proposition}
\begin{proof} (i) \,
Let $H: E \times I \to E$ be a homotopy, over $B,$ such that $H(x,0) = f(x)$ and $H(x,1) = g(x).$ We consider the path $\nu(t) = \omega(t) \ast H(x_0, t).$ Since $p(H(x_0, t)) = p(x_0)$ then $\nu$ and $\omega$ belong the same fiber, that is, $p^{-1}(p(x_0)).$ By induction defines $H^{n}(x,t) = H^{n-1}(H(x,t),t)$ for $n \geq 2.$ The homotopy $H^{n}$ between $f^{n}$ and $g^{n}$ is also over $B.$

The path $n(\nu)$ is homotopic to the path $n(w) \ast H^{n}(x_0, -)$ by a homotopy whose image is in the fiber $p^{-1}(p(x_0)).$ In fact, we consider the homotopy $G: I \times I \to E$ defined by $G(t,s) = H^{n-1}(H(x_0,t),s).$ We have that $H^{n}(x_0,t) = G(t,t)$ is homotopic to 
$G(0,s) \ast G(t,1)=$ $H^{n-1}(f(x_0),s) \ast g^{n-1} \circ H(x_0,t).$ 
Now we suppose that $(n-1)\nu$ is homotopic to $(n-1)\omega \ast H^{n-1}(x_0,t).$ We have $n\nu = (n-1)\nu \ast g^{n-1}(\nu),$ thus $n\nu$
is homotopic to $(n-1)\omega \ast H^{n-1}(x_0,t) \ast g^{n-1}(\omega) \ast g^{n-1} \circ H(x_0,t)=$ $n\omega \ast ( f^{n-1}(\omega) \ast H^{n-1}(x_0,t) \ast g^{n-1}(\omega)) \ast g^{n-1} \circ H(x_0,t).$ Using $H^{n-1}(x,t),$ is not difficult to see that 
$( f^{n-1}(\omega) \ast H^{n-1}(x_0,t) \ast g^{n-1}(\omega))$ is homotopic to $H^{n-1}(f(x_0),t).$ Therefore $n\nu$ is homotopic to $n\omega \ast H^{n-1}(f(x_0),t) \ast  g^{n-1} \circ H(x_0,t)$ which is homotopic to $n\omega \ast H^{n}(x_0,t).$ Now the statement follows by induction. 

From \cite[Page 8]{K-06} $H^{n}$ induces a fiber homotopy equivalence
$$ E_B(f^{n}) \stackrel{H^{n}_B}{\longrightarrow}  E_B(g^{n}) $$
defined by $H^{n}_B(x_0, \theta) = (x_0, \theta \ast H^{n}(x_0,t)),$ which preserves the essentiality of the Reidemeister classes. From this map and the map defined in Theorem \ref{th-rb}, we obtain a bijection $\Psi^{n}: \mathcal{R}_{B}(f^{n}; x_0, n(\omega)) \to \mathcal{R}_{B}(g^{n}; x_0, n(\nu)),$  which preserves the essentiality of the Reidemeister classes and such that the two diagrams
$$
\xymatrix{
	\mathcal{R}_{B}(f^{n}; x_0, n(\omega)) \ar[r]^-{\Psi^{n}} \ar[d]_-{[f^{\omega}]} & \mathcal{R}_{B}(g^{n}; x_0, n(\nu)) \ar[d]^-{[g^{\omega}]} \\
	\mathcal{R}_{B}(f^{n}; x_0, n(\omega)) \ar[r]^-{\Psi^{n}} & \mathcal{R}_{B}(g^{n}; x_0, n(\nu)),
} \hspace{2cm}
\xymatrix{
	\mathcal{R}_{B}(f^{m}; x_0, m(\omega)) \ar[r]^-{\Psi^{m}} \ar[d]_-{[\gamma_{m,n}]} & \mathcal{R}_{B}(g^{m}; x_0, m(\nu)) \ar[d]^-{[\gamma_{m,n}]} \\
	\mathcal{R}_{B}(f^{n}; x_0, n(\omega)) \ar[r]^-{\Psi^{n}} & \mathcal{R}_{B}(g^{n}; x_0, n(\nu)).
}
$$
are commutative. Therefore $N_B P_n(g) = N_B P_n(f).$

(ii) \, If $[\alpha]_{n}  \in \mathcal{R}_{B}(f^{n}; x_0, n(\omega))$ is an irreducible and essential class then there is an essential class $\mathcal{F}_n$ such that$\phi(\mathcal{F}_n) = [\alpha]_{n}.$ If $x \in \mathcal{F}_n$ then $f^{n}(x) = x$ and $f^{m}(x) \neq x$ for each positive divisor $m$ of $n$ which is different of $n.$ Thus different irreducible $f$-orbit of $\mathcal{F}_n$ do not have common periodic points. Now the statement follows from $(i)$ and Theorem \ref{th-rb}.  
\end{proof}

\begin{corollary}
 $\displaystyle  \sum_{m | n} N_B P_m(f) \leq MCF_n(f).$
\end{corollary}
\begin{proof}
Just see that the set $\{P_m(f); \,\, m | n \}$ defines a partition of $Fix(f^{n}).$
\end{proof}


\section{Some computations of $N_BP_n(f)$}

Motivated by the classic case, in this section, we will relate $N_BP_n(f)$ with the Nielsen number over $B$ in some particular cases.

\begin{definition}
A map $f: E \to E$ is said to be $n$-toral over $B$ if the following two  conditions hold:

\
 
(i) For every $m | n$ and each  $[\alpha]_{m}  \in \mathcal{R}_{B}(f^{m}; x_0, m(\omega)),$ $d(<[\alpha]_{m}>) = l(<[\alpha]_{m}>);$

\

(ii)  For all $m | n,$ $[\gamma_{m,n}]$ is injective.
\end{definition}

\begin{definition}
Let $f: E \to E$ be a fiber-preserving map over $B.$  We define, by induction:
$$A_{1}(f)=N_B(f) \ \ and \ \ A_{n}(f) = N_B(f^{n}) - \sum_{k|n, \ k<n} A_{k}(f).$$ 	
\end{definition}

\begin{theorem} \label{th-nbpn}
If $f: E \to E$ is an $n$-toral map over $B$ such that for every $m | n;$ $0 \neq N_B(f^{m}) = R_B(f^{m})$ then $$A_n(f) = N_BP_n(f).$$
\end{theorem}
\begin{proof}
Since $f$ is $n$-toral, $N_BP_n(f)$ is the cardinality of the irreducible and essential Reidemeister classes of $f^{n}.$  For every $m | n,$ since $0 \neq N_B(f^{m}) = R_B(f^{m})$, then all Reidemeister classes of $f^{n}$ are essential. 

Given $[\alpha]_{n}  \in \mathcal{R}_{B}(f^{n}; x_0, n(\omega))$, we suppose that $d([\alpha]_{n}) = m < n.$ This implies that $[\alpha]_n$ is reducible to $m.$ Thus, there exist $[\beta]_{m} \in \mathcal{R}_{B}(f^{m}; x_0, m(\omega))$ such that $[\gamma_{m,n}]([\beta]_{m}) = [\alpha]_{n}.$ Since $[\gamma_{m,n}]$ is injective then $\# Im([\gamma_{m,n}]) = N_B(f^{m}).$ By definition of the depth, $[\beta]_{m}$ is essential and irreducible in $\mathcal{R}_{B}(f^{m}; x_0, m(\omega)).$ Therefore the cardinality of Reidemeister classes of $f^{n}$ with depth $m < n$ is precisely; $ \displaystyle N_B(f^{m}) - \sum_{j|n, \ j<n} A_{j}(f) = A_m(f).$ 

Let $ 1 = k_1 < k_2 < k_3 < \cdots < k_{q} < n $ be the positive divisors of $n.$ Using the same argument above, then the cardinality of irreducible and essential Reidemeister classes of $f^{n}$ will be $N_B(f^{n}) - (A_{k_1} + A_{k_2} + \cdots + A_{k_q}),$ that is,  
$$N_BP_n(f) = N_B(f^{n}) - \sum_{k|n, \ k<n} A_{k}(f)  = A_n(f).$$
\end{proof}

The next result gives an explicit expression of $A_n(f)$ in terms of $N_B(f^{k}).$

\begin{theorem} \label{theorem-An}
Let $f: E \to E$ be a map over $B$ and $n \in \mathbb{N}.$ Let $n = p_1^{\alpha_1}p_2^{\alpha_2}\cdots p_t^{\alpha_t}$ be its prime factorization, where $t \geq 1,$ every $p_i$ prime and every $\alpha_i \geq 1.$ Then
$$ A_n(f) = \displaystyle \sum_{\substack{\alpha_{j}-1 \leq k_j \leq \alpha_j \\ 1  \leq j \leq t}}
(-1)^{(\alpha_1+\alpha_2+\cdots+\alpha_t)-(k_1+k_2+\cdots+k_t)} N_B(f^{p_1^{k_1}p_2^{k_2}\cdots p_t^{k_t}}).$$	
\end{theorem}	
\begin{proof}
	We have
	\begin{equation} \label{eq-ind-1}
		\begin{array}{lll}
			N_B(f^{p_1^{\alpha_1}p_2^{\alpha_2}\cdots p_t^{\alpha_t}}) - N_B(f^{p_1^{\alpha_1-1}p_2^{\alpha_2}\cdots p_t^{\alpha_t}}) & = & 
			\displaystyle \sum_{\substack{0 \leq i_j \leq \alpha_j \\
					1  \leq j \leq t}}A_{p_1^{i_1}p_2^{i_2}\cdots p_t^{i_t}}(f) -  \displaystyle\sum_{\substack{0 \leq l_j \leq \alpha_j, j \neq 1 \\ 0 \leq l_1 \leq \alpha_1-1 \\ 1  \leq j \leq t}} A_{p_1^{l_1}p_2^{l_2}\cdots p_t^{l_t}}(f) \\
			& = & \displaystyle\sum_{\substack{0 \leq i_j \leq \alpha_j \\
					2  \leq j \leq t}}A_{p_1^{\alpha_1}p_2^{i_2}\cdots p_t^{i_t}}(f). \\
		\end{array} 
	\end{equation}
	Replacing $\alpha_2$ by $\alpha_2-1$ in Equation \eqref{eq-ind-1} we obtain:
	\begin{equation} \label{eq-ind-2}
		\begin{array}{lll}
			N_B(f^{p_1^{\alpha_1}p_2^{\alpha_2-1}\cdots p_t^{\alpha_t}}) - N_B(f^{p_1^{\alpha_1-1}p_2^{\alpha_2-1}\cdots p_t^{\alpha_t}}) & = &
			\displaystyle\sum_{\substack{0 \leq i_j \leq \alpha_j, \,\, j \neq 2 \\ 0 \leq i_2 \leq \alpha_2-1}}A_{p_1^{\alpha_1}p_2^{i_2}\cdots p_t^{i_t}}(f). \\
		\end{array} 
	\end{equation}
	Subtracting Equation \eqref{eq-ind-2} from Equation \eqref{eq-ind-1} we get: 
	\begin{equation} \label{eq-ind-3}
		\begin{array}{lll}
			\displaystyle\sum_{\substack{0 \leq i_j \leq \alpha_j \\ 3 \leq j \leq t}}A_{p_1^{\alpha_1}p_2^{\alpha_2}p_3^{i_3}\cdots p_t^{i_t}}(f)
			& = & \displaystyle \sum_{\substack{\alpha_{j}-1 \leq k_j \leq \alpha_j \\
					1  \leq j \leq 2}}
			(-1)^{(\alpha_1+\alpha_2)-(k_1+k_2)} N_B(f^{p_1^{k_1}p_2^{k_2}p_3^{\alpha_3}\cdots p_t^{\alpha_t}}) \\
		\end{array} 
	\end{equation}
	We suppose that;
	\begin{equation} \label{eq-ind-4}
		\begin{array}{lll}
			\displaystyle\sum_{\substack{0 \leq i_j \leq \alpha_j \\ l+1 \leq j \leq t}}A_{p_1^{\alpha_1}p_2^{\alpha_2}\cdots p_l^{\alpha_l} p_{l+1}^{i_{l+1}}\cdots p_t^{i_t}}(f)
			& = & \displaystyle \sum_{\substack{\alpha_{j}-1 \leq k_j \leq \alpha_j \\
					1  \leq j \leq l}}
			(-1)^{(\alpha_1+\alpha_2+\cdots+\alpha_l)-(k_1+k_2+\cdots+k_l)} N_B(f^{p_1^{k_1}p_2^{k_2}\cdots p_l^{k_l} p_{l+1}^{\alpha_{l+1}}\cdots p_t^{\alpha_t}}) \\
		\end{array} 
	\end{equation}
	Replacing $\alpha_{l+1}$ by $\alpha_{l+1}-1$ in Equation \eqref{eq-ind-4} we obtain:
	\begin{equation} \label{eq-ind-5}
		\begin{array}{lll}
			\displaystyle\sum_{\substack{0 \leq i_j \leq \alpha_j, \\ j \neq l+1 \\ 0 \leq i_{l+1} \leq \alpha_{l+1}-1 \\ l+1 \leq j \leq t}}A_{p_1^{\alpha_1}p_2^{\alpha_2}\cdots p_l^{\alpha_l} p_{l+1}^{i_{l+1}}\cdots p_t^{i_t}}(f)
			& = & \displaystyle \sum_{\substack{\alpha_{j}-1 \leq k_j \leq \alpha_j \\
					1  \leq j \leq l}}
			(-1)^{(\alpha_1+\alpha_2+\cdots+\alpha_l)-(k_1+k_2+\cdots+k_l)} N_B(f^{p_1^{k_1}p_2^{k_2}\cdots p_l^{k_l} p_{l+1}^{\alpha_{l+1}-1}\cdots p_t^{\alpha_t}}) \\
		\end{array} 
	\end{equation}
	Subtracting Equation \eqref{eq-ind-5} from Equation \eqref{eq-ind-4} we get: 
	\begin{equation} \label{eq-ind-6}
		\begin{array}{lll}
			\displaystyle\sum_{\substack{0 \leq i_j \leq \alpha_j \\ l+2 \leq j \leq t}}A_{p_1^{\alpha_1}p_2^{\alpha_2}\cdots p_{l+1}^{\alpha_{l+1}} p_{l+2}^{i_{l+2}}\cdots p_t^{i_t}}(f)
			& = & \displaystyle \sum_{\substack{\alpha_{j}-1 \leq k_j \leq \alpha_j \\
					1  \leq j \leq l+1}}
			(-1)^{(\alpha_1+\alpha_2+\cdots+\alpha_{l+1})-(k_1+k_2+\cdots+k_{l+1})} N_B(f^{p_1^{k_1}p_2^{k_2}\cdots p_{l+1}^{k_{l+1}} p_{l+2}^{\alpha_{l+2}}\cdots p_t^{\alpha_t}}) \\
		\end{array} 
	\end{equation}
	Therefore, by induction the result follows.
\end{proof}


From the above results we have;

\begin{corollary} 
Let $f: E \to E$ be an $n$-toral map over $B$ such that for every $m | n;$ $0 \neq N_B(f^{m}) = R_B(f^{m}).$ Let ${\bf P}(n) = \{p(1), \cdots, p(k) \}$ be the set of all distinct primes which divides $n.$ Then,
$$N_BP_n(f) = \displaystyle \sum_{\tau \subseteq {\bf P}(n)}  (-1)^{\# \tau} N_B(f^{n \, : \, \tau}) $$
where $n \, : \, \tau = n \left( \displaystyle \prod_{p \in \tau} \, p  \right)^{-1}.$
\end{corollary}

Let $T$ be the torus considered as a trivial bundle with base and fiber $S^{1}.$ If for each $k|n$ the map $f^{k}: T \to T$ is not deformed to a fixed point free map over $S^{1},$ then we will prove that $f$ is $n$-toral.  

We consider $S^{1}$ as the unitary complex numbers subset, the torus $T=S^{1}\times S^{1}$, $p$ is the projection onto the second coordinate,  $S^{1}\to T \stackrel{p}{\to} S^{1}$  and $f:T\to T$ a map over $S^{1}$. So, $f$ is fiberwise homotopic to $f_{r,s}:T\to T$ defined by $f_{r,s}(x,y)=(x^{r}y^{s},y)$ for some $r,s \in\z,$ see \cite{G-L-V-V}. Moreover, given $n\in\z$, $n>1$, if we denote $\sigma(n,r)=\displaystyle{\sum_{i=0}^{n-1}r^{i}}$, then we have the following lemma:

\begin{lemma}
	
	If $f_{r,s}(x,y)=(x^{r}y^{s},y)$ then $f^{n}_{r,s}$ is given by $f^{n}_{r,s}(x,y)=\Big(x^{r^{n}}.y^{s.\sigma(n,r)},y\Big)$ for each $n\in \n.$ 
	
\end{lemma}

\begin{proof}
	In fact, by induction: $f^{2}_{r,s}(x,y)= ((x^{r}y^{s})^{r}.y^{s},y) = (x^{r^{2}}y^{s.r+s},y)$. In this sense, we note that $s.r+s=s.(r^{1}+r^{0})=s.\sigma(2,r)$, and it is same in general case as follows: $$\displaystyle{s+s.r.\sigma(n,r)=s.\Big(r^{0}+r.\sum_{i=0}^{n-1}r^{i}\Big)=s.\sum_{i=0}^{n}r^{i}=s.\sigma(n+1,r)}.$$ If we suppose $f^{n}_{r,s}(x,y)$ as above, we have:
	$$\begin{array}{cclcl}
		f^{n+1}_{r,s}(x,y) & = & f_{r,s}\Big(x^{r^{n}}.y^{s.\sigma(n,r)},y\Big) & = & \bigg(\Big(x^{r^{n}}.y^{s.\sigma(n,r)}\Big)^{r}.y^{s},y\bigg)\\
		& & & &\\
		& = & \Big(x^{r^{n+1}}.y^{s+s.r.\sigma(n,r)},y\Big) & = & \Big(x^{r^{n+1}}.y^{s.\sigma(n+1,r)},y\Big).\\
	\end{array}$$
\end{proof}

\begin{proposition}
	
	Let $T = S^{1} \times S^{1} \stackrel{p_2}{\longrightarrow} S^{1} $ be the trivial bundle, $f: T \to T$ a map over $S^{1}$ and $n$ a positive integer. If $f^{k}$ can not be deformed over $S^{1}$ to a fixed point free map for all $k | n,$ then $f$ is a $n$-toral map over $S^{1}.$ 
\end{proposition}
\begin{proof}
	We have that $f$ is homotopic over $S^{1}$ to a map $f_{r,s}$ for some $r,s \in \mathbb{Z}.$ Thus in the proof is enough to consider the map $f_{r,s}.$
	Since $f^{k}_{r,s}$ is non fiberwise homotopic to a fixed point free map for each divisor $k$ of $n$, then there exists a fixed point $(x_0,y_0)$ of $f_{r,s}$ and its constant path $\omega$ at $(x_0,y_0)$.
	So, for all $k\in\z$, $k(\omega)$ is the constant path at $(x_0,y_0)$ and given $[\theta]\in \pi_{1}(S^{1},x_0)$ we will denote by the same symbol $[\theta]= [(\theta_1,\theta_2)] = i_{\#}([\theta])$ the inclusion of $\theta$ in $\pi_1(T, (x_0,y_0)),$ where $\theta_{2}$ is the constant path at $y_{0}.$ We will write $\theta_{2}=c_{y_{0}}$ and $[\theta]=\left[\left(\theta_{1},c_{y_{0}}\right)\right]$. 
	
	Observe that the map $f_{r,s}^{k\omega}:\pi_{1}(T,(x_0,y_0))\to\pi_{1}(T,(x_0,y_0))$ is given by $f_{r,s}^{k\omega}([\theta])=[f^{k}_{r,s}(\theta_1,c_{y_{0}})]=\left[\left(\theta_{1}^{r^{k}},c_{y_{0}}\right)\right]$. Therefore,
	
	$$ \begin{array}{ccl} \gamma_{k,n}([\theta]) & = & [\theta \ast f^{k}_{r,s}(\theta) \ast f^{2k}_{r,s}(\theta) \ast \cdots \ast f^{n-k}_{r,s}(\theta)]\\
		& = & \left[\left(\theta_1,c_{y_{0}}\right) \ast \left(\theta_{1}^{r^{k}},c_{y_{0}}\right) \ast \left(\theta_{1}^{r^{2k}},c_{y_{0}}\right) \ast \cdots \ast \left(\theta_{1}^{r^{n-k}},c_{y_{0}}\right)\right]\\
		& = & \left[\left(\theta_{1}^{1+r^{k}+r^{2k}+\dots+r^{n-k}},c_{y_{0}}\right)\right].
	\end{array}$$
	
	Let $1+r^{k}+r^{2k}+\dots+r^{n-k}=A(n,k)\in\z.$ Since $f^{k}_{r,s}$ can not be deformed to a fixed point free map over $S^{1}$ then $A(n,k) \neq 0.$ We have $\gamma_{k,n}([\theta]) = \left[\left(\theta_{1}^{A(n,k)},c_{y_{0}}\right)\right]=\left[\left(\theta_{1},c_{y_{0}}\right)^{A(n,k)}\right]$. Suppose that $\gamma_{k,n}([\theta])$ and $\gamma_{k,n}([\theta'])$ are Reidmeister related over $S^{1}$, so there exists $[c]=[(c_{1},c_{2})] \in \pi_1(T, (x_0,y_0))$ such that we can lift the homotopy $H: I \times I \to S^{1}$ defined by $H(t,s) = p(c(t))=c_{2}(t)$ to a homotopy $\widetilde{H}: I \times I \to T$ such that
	
	\begin{center}
		\setlength{\unitlength}{0,8cm}
		\begin{picture}(6.7,3.3)
			\thicklines
			\put(1.5,0.5){\line(1,0){3}}
			\put(1.5,2.5){\line(1,0){3}}
			\put(1.5,0.5){\line(0,1){2}}
			\put(4.5,0.5){\line(0,1){2}}
			\put(2.9,0){$c$}
			\put(2.5,2.8){$f^{n}(c)$}
			\put(-1.4,1.4){$\left(\theta_{1},c_{y_{0}}\right)^{A(n,k)}$}
			\put(4.8,1.4){$\left(\theta_{1}',c_{y_{0}}\right)^{A(n,k)}$}
		\end{picture}
	\end{center}
	
	Thus $\left(\theta_{1}',c_{y_{0}}\right)^{A(n,k)}$ is homotopic to $c^{-1}\ast \left(\theta_{1},c_{y_{0}}\right)^{A(n,k)}\ast f_{r,s}^{n}(c)$ and
	$$\begin{array}{ccl}
		c^{-1}\ast \left(\theta_{1},c_{y_{0}}\right)^{A(n,k)}\ast f^{n}(c)
		&=& \left(c_{1}^{-1},c_{2}^{-1} \right) \ast \left(\theta_{1}^{A(n,k)},c_{y_{0}}\right)\ast \left(c_{1}^{r^{n}}\ast c_{2}^{s.\sigma(n,r)},c_{2}\right)\\
		&=& \left(c_{1}^{-1} \ast \theta_{1}^{A(n,k)} \ast c_{1}^{r^{n}}\ast c_{2}^{s.\sigma(n,r)},c_{2}^{-1} \ast c_{y_{0}} \ast c_{2} \right) \\
		&=& \left(\theta_{1}^{A(n,k)} \ast c_{1}^{r^{n}-1}\ast c_{2}^{s.\sigma(n,r)}, c_{2}^{-1} \ast c_{y_{0}} \ast c_{2} \right) \\
		&=& \left(\theta_{1}^{A(n,k)} \ast c_{1}^{A(n,k)(r^{k}-1)}\ast c_{2}^{A(n,k)s.\sigma(k,r)}, c_{2}^{-1} \ast c_{y_{0}} \ast c_{2} \right) \\
		&=& \left(\left(\theta_{1} \ast c_{1}^{r^{k}-1}\ast c_{2}^{s.\sigma(k,r)}\right)^{A(n,k)}, c_{2}^{-1} \ast c_{y_{0}} \ast c_{2} \right)\\
		&=& \left(c_{1}^{-1}\ast\theta_{1} \ast c_{1}^{r^{k}}\ast c_{2}^{s.\sigma(k,r)}, c_{2}^{-1} \ast c_{y_{0}} \ast c_{2} \right)^{A(n,k)}.
	\end{array}$$
	
	We have $\left[\left(\theta_{1}',c_{y_{0}}\right)^{A(n,k)}\right] = \left[\left(c_{1}^{-1}\ast\theta_{1} \ast c_{1}^{r^{k}}\ast c_{2}^{s.\sigma(k,r)}, c_{2}^{-1} \ast c_{y_{0}} \ast c_{2} \right)^{A(n,k)}\right].$ 
	Using the isomorphism $\pi_{1}(S^{1}, x_0) \cong \z$ and which $A(n,k) \neq 0$ we obtain;
	$$
	\begin{array}{lll}
		\left[\theta'\right] & = & \left[\left(\theta_{1}',c_{y_{0}}\right)\right] \\
		& = & \left[\left(c_{1}^{-1}\ast\theta_{1} \ast c_{1}^{r^{k}}\ast c_{2}^{s.\sigma(k,r)}, c_{2}^{-1} \ast c_{y_{0}} \ast c_{2} \right)\right] \\
		& = & \left[\left(c_{1}^{-1}\ast\theta_{1} \ast c_{1}^{r^{k}}\ast c_{2}^{s.\sigma(k,r)}, c_{2}^{-1} \ast c_{y_{0}} \ast c_{2} \right)\right] \\
		& = & \left[c^{-1}\ast \theta \ast f^{k}(c)\right] \\
	\end{array}
	$$  
	
	Since $\pi_2(S^{1}) = 0$ and from \cite[Equation (3.7)]{G-K-09} we have that $[c]\ast_{S^{1}}[\theta]=[\theta'].$ Therefore,  $[\theta]=[\theta']$ in $ \mathcal{R}_{S^{1}}\left(f^{k}; (x_0,y_0), k(\omega)\right),$ and $\left[\gamma_{k,n}\right]$ is injective, under our hypothesis.
\end{proof}

\begin{proposition}
Let $T = S^{1} \times S^{1} \stackrel{p_2}{\to} S^{1} $ be the trivial bundle and $f: T \to T$ a map over $S^{1}$ such that $f \simeq_{S^{1}} f_{1,s}$ and $s \neq 0$. If $n=p_{1}^{\alpha_{1}}\dots p_{l}^{\alpha_{l}}$ is its prime factorization, then $$N_{S^{1}}P_n(f) = |s|p_{1}^{\alpha_{1}-1}\dots p_{l}^{\alpha_{l}-1}(p_{1}-1)\dots(p_{l}-1).$$	
\end{proposition}
\begin{proof} Firstly, we observe that $N_{S^{1}}(f) = gcd\{1-1,s\}=gcd\{0,s\}=|s|$ and therefore $N_{S^{1}}(f^{m})=|\sigma(m,1)|N_{S^{1}}(f)=m|s| \neq 0$, for all $m\in\n$. From \cite[Theorem 1.3]{G-K-09}, in this case, we have $N_{S^{1}}(f^{m}) = R_{S^{1}}(f^{m}) = MCF_m(f) $ for every $m | n.$ From Theorem \ref{th-nbpn}  we have:
$$ \begin{array}{ccl}
N_{S^{1}}P_n(f) & = & \displaystyle \sum_{\substack{\alpha_{j}-1 \leq k_j \leq \alpha_j \\ 1  \leq j \leq l}} 
	(-1)^{(\alpha_1+\alpha_2+\cdots+\alpha_l)-(k_1+k_2+\cdots+k_l)} N_{S^{1}}(f^{p_1^{k_1}p_2^{k_2}\cdots p_l^{k_l}}) \\
	& = & \displaystyle\sum_{\substack{\alpha_{j}-1 \leq k_j \leq \alpha_j \\ 1  \leq j \leq l}}
	(-1)^{(\alpha_1+\alpha_2+\cdots+\alpha_l)-(k_1+k_2+\cdots+k_l)} p_1^{k_1}p_2^{k_2}\cdots p_l^{k_l}.|s| \\
	& = & |s|.p_1^{\alpha_1 -1}p_2^{\alpha_2 -1}\cdots p_l^{\alpha_l -1}.
	\displaystyle\sum_{\substack{\alpha_{j}-1 \leq k_j \leq \alpha_j \\1  \leq j \leq l}}
	(-1)^{(\alpha_1+\alpha_2+\cdots+\alpha_l)-(k_1+k_2+\cdots+k_l)} p_1^{\alpha_1 - k_1}p_2^{\alpha_2 -k_2}\cdots p_l^{\alpha_l -k_l}.
	\end{array}$$		

We observe that, for any $1\leq u< l$, we have:
$$	\displaystyle\sum_{\substack{\alpha_{j}-1 \leq k_j \leq \alpha_j \\ u  \leq j \leq l}}
	(-1)^{(\alpha_u+\cdots+\alpha_l)-(k_u+\cdots+k_l)} p_u^{\alpha_u -k_u}\cdots p_l^{\alpha_l -k_l}$$
$$	= p_u.\displaystyle\sum_{\substack{\alpha_{j}-1 \leq k_j \leq \alpha_j \\ u+1  \leq j \leq l}}
	(-1)^{(\alpha_{u+1}+\cdots+\alpha_l)-(k_{u+1}+\cdots+k_l)} p_{u+1}^{\alpha_{u+1} -k_{u+1}}\cdots p_l^{\alpha_l -k_l}$$
$$	-\displaystyle\sum_{\substack{\alpha_{j}-1 \leq k_j \leq \alpha_j \\ u+1  \leq j \leq l}}
	(-1)^{(\alpha_{u+1}+\cdots+\alpha_l)-(k_{u+1}+\cdots+k_l)} p_{u+1}^{\alpha_{u+1} -k_{u+1}}\cdots p_l^{\alpha_l -k_l}$$		
$$	= (p_u-1).\displaystyle\sum_{\substack{\alpha_{j}-1 \leq k_j \leq \alpha_j \\ u+1  \leq j \leq l}}
	(-1)^{(\alpha_{u+1}+\cdots+\alpha_l)-(k_{u+1}+\cdots+k_l)} p_{u+1}^{\alpha_{u+1} -k_{u+1}}\cdots p_l^{\alpha_l -k_l}.$$		
Then,
$$ \begin{array}{ccl}
N_{S^{1}}P_n(f) & = & |s|.p_1^{\alpha_1 -1}p_2^{\alpha_2 -1}\cdots p_l^{\alpha_l -1}.(p_1 -1).(p_2 -1)\dots(p_l -1).
	\end{array}$$	
		
\end{proof}


\end{document}